\documentclass[11pt]{amsart}
\usepackage{amssymb}
\newcommand{\R}{{\mathbb R}}
\newcommand{\I}{{\mathbf I}}

\newcommand{\al}{\alpha}
\newcommand{\ap}{\approx}
\newcommand{\dy}{\mathcal{D}}

\newcommand{\w}{\mathbf{W}_{1,p}}
\newcommand{\wa}{\mathbf{W}_{\alpha,p}}

\newcommand{\rat}{\mathbf{I}_{2\alpha}}
\newcommand{\ra}{\mathbf{I}_{\alpha}}

\newcommand{\rtw}{\mathbf{I}_2}

\newcommand{\M}{M^{+}(\R^n)}
\begin{document}
\title[Pointwise estimates of solutions] {Pointwise estimates of Brezis--Kamin type \\ for solutions of sublinear elliptic equations }
\author{Dat T. Cao}
\address{Department of Mathematics, University of Tennessee, Knoxville, 
TN 37996, USA}
\email{dcao4@utk.edu}
\author{ Igor E. Verbitsky}
\address{Department of Mathematics, University of Missouri, Columbia, MO  65211, USA}
\email{verbitskyi@missouri.edu}
\begin{abstract}
We study quasilinear elliptic equations of the type 
$$
-\Delta_pu=\sigma \, u^q \quad \text{in} \, \, \,  {\R}^n,
$$
where  $\Delta_p u=\nabla \cdot(\nabla u |\nabla u|^{p-2})$ is the $p$-Laplacian (or a more general $\mathcal{A}$-Laplace operator $\text{div} \,  \mathcal{A}(x, \nabla u)$), $0<q < p-1$, and $\sigma \ge 0$ is an arbitrary locally integrable function or measure on $\R^n$.  

We  obtain  necessary and sufficient conditions for the existence of  positive solutions (not necessarily bounded) which satisfy global pointwise estimates 
of Brezis--Kamin type given in terms of Wolff potentials. 
Similar problems with the fractional Laplacian $(-\Delta )^{\alpha}$  for $0<\alpha<\frac{n}{2}$ are treated as well,  including explicit estimates for  radially symmetric 
$\sigma$. Our  results are new even in the classical case $p=2$ and $\alpha=1$. 
\end{abstract}
\maketitle
\numberwithin{equation}{section}
\newtheorem{theorem}{Theorem}[section]
\newtheorem{lemma}[theorem]{Lemma}
\newtheorem{rmk}[theorem]{Remark}
\newtheorem{cor}[theorem]{Corollary}
\newtheorem{prop}[theorem]{Proposition}
\newtheorem{defn}[theorem]{Definition}
\allowdisplaybreaks
\section{Introduction} \label{intro}
We study quasilinear equations of the type 
\begin{equation}\label{eq1}
\left\{ \begin{array}{ll}
-\Delta_pu=\sigma \, u^q & \text{in} \, \, \,  {\R}^n,\\&\\
\displaystyle{\liminf_{x \to \infty} \, u(x)}=r,  & \, \, u>0, 
\end{array} \right.
\end{equation}
where $\Delta_p u=\nabla \cdot(\nabla u |\nabla u|^{p-2})$ is the $p$-Laplacian, $0<q<p-1$, $r \ge 0$, and $\sigma \in \M$; here $\M$  denotes the class of all nonnegative locally finite Borel measures on $\R^n$.  Our 
main results are new, in particular,  for  nonnegative $\sigma \in L^1_{\rm loc}(\R^n)$.   

We prove the existence of distributional solutions to \eqref{eq1}  and obtain sharp global pointwise bounds of solutions of Brezis--Kamin type (see \cite{BK} for bounded solutions in the case $p=2$) in terms of Wolff potentials. Necessary and sufficient conditions on $\sigma$  for the existence of solutions to \eqref{eq1} which satisfy such estimates will be presented.  

We also obtain similar results for the fractional Laplace equation, 
\begin{equation}\label{frac}
\left\{ \begin{array}{ll}
(-\Delta)^{\al}u=\sigma \, u^q & \text{in} \, \, \,  {\R}^n,\\&\\
\displaystyle{\liminf_{x \to \infty} \, u(x)}=r,  & \, \, u>0, 
\end{array} \right.
\end{equation}
where $0<q<1$ and  $0<\al<\frac{n}{2}$. In particular, for radially symmetric $\sigma$, explicit conditions which ensure  the existence of radial solutions together with sharp bilateral bounds of solutions are given.

In the classical case $p=2$, the sublinear problem \eqref{eq1}, or equivalently \eqref{frac} with $\al=1$,  was studied in \cite{BK}, where a necessary and sufficient condition for the existence of  {\it bounded}  solutions was given, together with uniqueness. Of specific 
interest to us are the following global pointwise estimates of bounded solutions obtained in \cite{BK}: 
\begin{equation} \label{BK-est} 
 c^{-1} \, ( \mathbf{I}_2 \sigma)^{\frac{1}{1-q}}\le u \le c\, \mathbf{I}_2  \sigma,
\end{equation}
where $c$ is a positive constant, and $\mathbf{I}_2 \sigma  = (-\Delta)^{-1} \sigma $ is the Newtonian potential of $\sigma$. 

The fractional Laplace equation \eqref{frac} 
 was studied recently by Punzo and  Terrone  \cite{PT}. They considered bounded solutions and obtained similar 
pointwise estimates  in terms of Riesz potentials $\rat \sigma  = (-\Delta)^{-\alpha} \sigma $ under the restriction $0<\al<\min\{1,\frac{n}{4}\}$. 

Recently, we characterized the existence of finite energy solutions \cite{CV1} and arbitrary distributional solutions \cite{CV2}, and gave matching upper and lower pointwise estimates of solutions using new intrinsic potentials of Wolff 
type, both for equation \eqref{eq1} with $0<q<p-1$, and \eqref{frac} 
with $0<q<1$ and $0<\al<\frac{n}{2}$. However, the definition and 
applications of intrinsic Wolff potentials are complicated. 

The main goal of this paper is to prove the existence of a broad class of solutions, not necessarily bounded, and obtain Brezis--Kamin type estimates similar to \eqref{BK-est} in terms of the usual Wolff potentials defined in \cite{HW} (see also \cite{AH}, \cite{Maz11}), or classical Riesz potentials in the case $p=2$ or $\alpha=1$,  both for  equation \eqref{eq1} with $0<q<p-1$, and \eqref{frac} with $0<q<1$ and $0<\al<\frac{n}{2}$.

We remark that equation \eqref{frac} is equivalent to the following integral equation 
involving Riesz potentials, 
\begin{equation} \label{sub-integral}
u=\rat(u^qd\sigma) + r, \quad u >0. 
\end{equation}

It was shown in \cite{CV2} that if there exists a nontrivial solution $u$ to the integral inequality
\begin{equation} \label{intgral-ine}
u \ge \rat(u^qd\sigma) \quad \text{in} \, \, \R^n, \quad u \in L^q_{\rm loc}(\R^n, d \sigma), 
\end{equation}
then $u$ satisfies the lower bound 
\begin{equation} \label{lowest}
u(x) \ge c \big(\rat\sigma(x)\big)^{\frac{1}{1-q}}, \quad  x \in \R^n,
\end{equation}
where  $c=c(n,q)>0$. Therefore, a necessary condition for the existence of a solution to \eqref{sub-integral} is obviously that $\rat\sigma \not \equiv \infty$, or equivalenty, 
\begin{equation}\label{Rieszfinite}
\int_1^{\infty} \frac{\sigma(B(0, t))}{t^{n- 2\al}}\frac{dt}{t} < +\infty.
\end{equation}
Here and below  $B(x,t)$ denotes a ball of radius $t$ centered at $x \in \R^n$. 

We will see that in fact $c_0 \left(\rat\sigma\right)^{\frac{1}{1-q}}$ is a subsolution to \eqref{sub-integral} provided the constant  $c_0>0$ is small enough, and  
\eqref{Rieszfinite} holds. In other words, in order to establish the existence of a solution to \eqref{sub-integral}, it  suffices to find a supersolution, 
which  is the main goal of this paper.

It was also noticed in \cite{CV2} that if there exists a solution to \eqref{intgral-ine}, then  $\sigma $ must be absolutely continuous with respect to $\text{cap}_{\al,2}$, which is the $(\al,2)$-capacity defined by \eqref{Rieszcapa}. In some of our results we will assume that $\sigma$ satisfies the following stronger condition 
\begin{equation}\label{cap2-cond}
\sigma(E) \le c(\sigma) \, \text{cap}_{\al, 2}(E), \, \,  \text{ for all compact sets }  \, E \subset {\R}^n,
\end{equation}
where $c(\sigma)>0$ is a constant which does not depend on $E$. Capacity conditions of this type  were introduced and applied to various problems 
by V. Maz'ya (see \cite{AH}, \cite{Maz11}, \cite{V1}). We will show below 
(see Lemma \ref{wolff-capacity}) that this condition is equivalent, 
for every $s>0$, to
\begin{equation} \label{riesz-integral}
\int_E (\rat\sigma_E)^s d\sigma \le c \, \sigma(E), 
\end{equation}
 where $\sigma_E$ denotes the restriction of  $\sigma$ 
 to a compact set $E$. This implies, in particular, that for all balls $B$, 
 or cubes in place of balls, 
\begin{equation} \label{riesz-ball}
\int_B (\rat\sigma_B)^s d\sigma \le c \, \sigma(B).  
\end{equation}

Estimates of this type for $s \ge 1$ were considered in \cite{Maz11}, \cite{V1}, \cite{JV1}, \cite{JV2}, and for $s>\frac{1}{2}$ in \cite{NTV03}.  
Such estimates for all $s>0$ are new, and allow us to find a supersolution and obtain upper estimates of solutions to \eqref{sub-integral}. 

Condition \eqref{cap2-cond}, or equivalently \eqref{riesz-ball} 
on cubes  for some $s>0$, 
 together with  \eqref{Rieszfinite}, ensures the existence of a solution $u$ to \eqref{sub-integral} which satisfies the following two-sided estimates if $r>0$: 
\begin{equation} \label{twosdinhomo} c^{-1} (r + \rat\sigma)^{\frac{1}{1-q}}\le u \le c (r + \rat\sigma)^{\frac{1}{1-q}},
\end{equation}
where $c=c(n,q,r,c(\sigma))>0$. It also yields the existence of a solution $u$ to the ground state problem \eqref{sub-integral} with $r=0$ such that \begin{equation} \label{twosdhomo} c^{-1} \, ( \rat\sigma)^{\frac{1}{1-q}}\le u \le c\, \left(\rat\sigma + \left(\rtw\sigma\right)^{\frac{1}{1-q}} \right).
\end{equation}

These estimates are sharp as indicated in \cite{BK}. Indeed, suppose that $\mathbf{I}_{2 \alpha}\sigma \in L^{\infty}(\R^n)$ as in \cite{BK}. Then by a well known result (see \cite{AH}, \cite{Maz11}) this implies that \eqref{cap2-cond}  holds. Therefore, there exists a solution $u$ to \eqref{sub-integral} with $r=0$ satisfying \eqref{twosdhomo}. Since $\mathbf{I}_{2 \alpha} \sigma \in L^{\infty}(\R^n)$, this yields 
$$ c^{-1} \, ( \mathbf{I}_{2 \alpha} \sigma)^{\frac{1}{1-q}}\le u\le c\, \mathbf{I}_{2 \alpha} \sigma,$$ 
which coincides with the Brezis--Kamin estimate \eqref{BK-est} in the case $\alpha=1$. However,  condition \eqref{cap2-cond} with 
$\alpha=1$ is in general weaker than the condition $\mathbf{I}_{2} \sigma \in L^{\infty}(\R^n)$ imposed in \cite{BK}, and is applicable to singular (unbounded) solutions.

Obviously, conditions \eqref{cap2-cond} and \eqref{riesz-ball} do not depend on the ``sub-critical'' growth rate $q$ at all. As we will demonstrate below, condition \eqref{riesz-ball} with $s=\frac{q}{1-q}$ can be relaxed to the following condition which does depend on $q$: 
\begin{equation} \label{weaker0}
\int_{B(x, \frac{r}{2})} \left( \rat\sigma\right )^{\frac{q}{1-q}}d\sigma \le c\, \sigma(B(x,r)) \left( 1 +\int_r^{\infty}  \frac{\sigma(B(x,t))}{t^{n-2\al}} \frac{dt}{t}\right)^{\frac{q}{1-q}}. 
\end{equation}

It is easy to see that  \eqref{weaker0} yields the following pointwise condition 
on $\sigma$:  
\begin{equation} \label{newcond0}
 \rat\left(\left( \rat\sigma\right )^{\frac{q}{1-q}}d\sigma\right) \le \kappa \left( \rat\sigma  + \left( \rat\sigma\right )^{\frac{1}{1-q}}\right) ,
\end{equation}
where $\kappa=\kappa(\al,n,q, c)>0$. 

As we will demonstrate below, condition \eqref{newcond0} combined with \eqref{Rieszfinite}, is actually \textit{necessary and sufficient} for the existence of a solution $u$ to  \eqref{sub-integral} with $r=0$ satisfying the Brezis--Kamin type 
estimates \eqref{twosdhomo}.

Both conditions \eqref{weaker0} and \eqref{newcond0} are  weaker than the capacity condition \eqref{cap2-cond}. It is easy to find $\sigma > 0$ such that  \eqref{weaker0} and hence \eqref{newcond0} hold, but \eqref{cap2-cond} fails. For instance,  one can let $\sigma(x)= \frac{1}{|x|^{s}}$ on the ball $B(0,1)$ and  zero outside $B(0,1)$, where $2\alpha<s<n-(n- 2 \alpha)q$. 

In fact, when $\sigma$ is radially symmetric, we will characterize condition \eqref{newcond0} as follows (see Proposition~\ref{radpro} in 
Sec.~\ref{radial} below): 
\begin{equation}\label{char}
\limsup_{|x| \to 0} \frac{\frac{1}{|x|^{(n-2\al)(1-q)}}\int_{|y|<|x|} \frac{d\sigma(y)}{|y|^{(n-2\al)q}} }{ \int_{|y| \ge |x|} \frac{d\sigma(y)}{|y|^{n-2\al} }} < \infty. 
\end{equation}

Also, for  radially symmetric $\sigma$, we obtain necessary and sufficient conditions for the existence of a solution $u$ to \eqref{sub-integral} with $ r=0$ in the following form:
\begin{equation} \label{radialcond}
\int_{|y|<1} \frac{d\sigma(y)}{|y|^{(n-2\al)q}} < \infty \text { and } \int_{|y|\ge 1} \frac{d\sigma(y)}{|y|^{n-2\al}}  < \infty.
\end{equation}
Moreover, such a solution $u$ satisfies matching lower and upper
estimates (see Theorem~\ref{radpro} below): 
\begin{equation}\label{radlow}
u(x) \approx \, \frac{1}{
|x|^{n-2\al}}\left(\int_{|y|<|x|} \frac{d\sigma(y)}{|y|^{(n-2\al)q}}\right)^{\frac{1}{1-q}} + \left(\int_{|y|\ge|x|} \frac{d\sigma(y)}{|y|^{n-2\al}}\right)^{\frac{1}{1-q}}.
\end{equation}
Here $A\approx B$ means $c_1 A \le B \le c_2 A$, where $c_1$ and $c_2$ are  positive constants. \smallskip

Before stating our main results for the quasilinear problem \eqref{eq1}, we need to discuss assumptions analogous to \eqref{Rieszfinite} and \eqref{cap2-cond} for the $p$-Laplacian. The Wolff potential $\wa\sigma$  is defined (\cite{HW}), for $1<p<\infty$ and $ 0<\al < \frac{n}{p}$, by 
$$\wa\sigma(x)=\int_0^{\infty} \left(\frac{\sigma(B(x,t))}{t^{n-\al p}}\right)^{\frac{1}{p-1}}\frac{dt}{t}. $$
 Let $\alpha=1$, and suppose that $\w\sigma \not \equiv \infty$, or equivalently,
\begin{equation} \label{Wolfinite}
\int_1^{\infty} \left(\frac{\sigma(B(0,t))}{t^{n- p}}\right)^{\frac{1}{p-1}}\frac{dt}{t} < \infty.
\end{equation} 
Suppose additionally that $\sigma$ satisfies the capacity condition
\begin{equation}\label{capacitycond}
\sigma(E) \le C(\sigma) \text{ cap}_p(E) \text{ for all compact sets }  E \subset {\R}^n,
\end{equation}
where $C(\sigma)>0$ and cap$_p(\cdot)$ stands for the $p$-capacity defined by
\begin{equation} \label{pcapacity}
\text{cap}_p(E)=\inf \{||\nabla f||^{p}_{L^p} : f \ge 1 \text{ on $E$, } f \in C_0^{\infty}(\R^n)\}, \quad E \subset \R^n,
\end{equation}
for compact sets $E$. 
 
We are now ready to state our first main result for equation \eqref{eq1} with $r>0$. 
\begin{theorem}\label{continhomothm}
Let $1<p<n, 0<q<p-1$, and $r >0$. Let $\sigma \in M^{+}(\R^n)$. If both \eqref{Wolfinite} and \eqref{capacitycond} hold, then there exists a distributional solution $u \in W_{\rm loc}^{1, p}(\R^n)$ to \eqref{eq1}  such that 
$$c^{-1} \Bigl( r + \w\sigma \Bigr)^{\frac{p-1}{p-1-q}} \le u \le c \Bigl( r + \w\sigma \Bigr)^{\frac{p-1}{p-1-q}} ,$$
where  $W_{\rm loc}^{1,p}(\R^n)$ is the usual local Sobolev space and $c>0$ depends only on $n,p,q,r$, and $C(\sigma)$. If $p \ge n$, 
then there are no nontrivial solutions on $\R^n$. 
\end{theorem}
We observe that in this case we have matching upper and lower estimates of the solution $u$. 

Our next theorem is concerned with the ground state problem \eqref{eq1} 
with $r=0$.
\begin{theorem}\label{conthomothm}
Let $1<p<n$ and $0<q<p-1$. Let $\sigma \in M^{+}(\R^n)$. If both \eqref{Wolfinite} and \eqref{capacitycond} hold, then there exists a mimimal $p$-superharmonic solution $u \in W_{\rm loc}^{1,p}(\R^n)$ to  \eqref{eq1} with $r=0$ such that 
\begin{equation} \label{two-sided}
c^{-1} \Bigl( \w\sigma \Bigr)^{\frac{p-1}{p-1-q}} \le u \le c \left(\w\sigma+\Bigl( \w\sigma \Bigr)^{\frac{p-1}{p-1-q}} \right),
\end{equation}
where $c=c(n,p,q,C(\sigma)) > 0$. In the case $p \ge n$ there are no nontrivial solutions on $\R^n$. 
\end{theorem}

See \cite{KM}, \cite{HKM} for the definition of $p$-superharmonic solutions, or equivalently locally renormalized solutions discussed in \cite{BV}, \cite{KKT}.

In the next theorem, we give a necessary and sufficient condition for the existence of a solution to \eqref{eq1} satisfying \eqref{two-sided}. 
\begin{theorem}\label{newthm}
Let $1<p<n, 0<q<p-1$ and $\sigma \in M^{+}(\R^n)$. Then there exists a $p$-superharmonic solution $u$  to \eqref{eq1} with $r=0$ satisfying  \eqref{two-sided} if and only if there exists a constant $\kappa >0$ such that
\begin{equation} \label{newcond}
 \w\left(\left( \w\sigma\right )^{\frac{(p-1)q}{p-1-q}}d\sigma\right) \le \kappa \left( \w\sigma  + \left( \w\sigma\right )^{\frac{p-1}{p-1-q}}\right) < \infty \quad \text{a.e.}
\end{equation}
\end{theorem}

The plan of the paper is as follows. 
In Section \ref{capWolff}, we study the equivalence between the capacity condition \eqref{cap2-cond} and condition \eqref{riesz-ball} for all $s>0$. In Section \ref{solnintegral}, we study integral equations closely related to \eqref{eq1}. Section \ref{proofthm} is devoted to a proof of our main results regarding equation \eqref{eq1}. In Section \ref{radial}, we consider the case where $\sigma$ is radially symmetric, and give an explicit condition for the existence of a radial solution to \eqref{frac}, together with matching bilateral pointwise estimates of such a solution. We also provide a proof of  \eqref{char}, and conclude with an example where there is a solution to \eqref{sub-integral}, but condition \eqref{newcond0} fails, and 
consequently the upper estimate of Brezis--Kamin type in \eqref{twosdhomo} is no longer true.

\section{Capacity condition and Wolff potential estimates } \label{capWolff}

For $p >1$, we define the $(\al,p)$-capacity of a subset $E \subset \R^n$ by 
\begin{equation} \label{Rieszcapa}
\text{cap}_{\alpha,p}(E)=\inf\{||f||^p_{L^p(\R^n)} : \I_{\alpha}f \ge 1 \text { on } E,\quad f \in{L^p_+(\R^n)} \},
\end{equation}
where   $\I_{\alpha}\sigma$ is the Riesz potential of order $\alpha$ defined for $0< \alpha <n$ by
\begin{equation} \label{Rieszpote}
\I_{\alpha}\sigma(x)=\int_0^{\infty} \frac{\sigma(B(x,r))}{r^{n-\alpha}}\frac{dr}{r}, \quad x \in \R^n.
\end{equation}
Notice that $\rm{cap}_{1,p}(E) \approx \rm{cap}_p(E)$ for all compact sets $E$ (see \cite{AH}).

\begin{lemma} \label{wolff-capacity} Let $\sigma \in M^{+}(\R^n)$ and suppose that
\begin{equation}\label{capcond}
\sigma(E) \le C\, {\rm cap}_{\alpha,p}(E),\quad \text { for all compact sets } E \subset R^n.
\end{equation}
Then, for every $s>0$, the following inequality holds: 
\begin{equation} \label{wolffintegral}
\int_E (\wa\sigma_E)^s d\sigma \le c \, \sigma(E), \quad \text { for all compact sets } E. 
\end{equation}
 Conversely, if \eqref{wolffintegral} holds for a given $s>0$, then \eqref{capcond} holds;  consequently, \eqref{wolffintegral} holds for every $s>0$.
\end{lemma}
\begin{proof}
Condition  \eqref{capcond} is known to be equivalent to the non-capacitary 
condition (see \cite{Maz11}, \cite{V1})
\begin{equation}\label{KS}
\int_{\R^n} (\ra \sigma_E)^{p'}dx \le c \,\sigma(E),
\end{equation}
for all compact sets $E \subset \R^n$. By Wolff's inequality (\cite{AH}), we obtain 
\begin{equation} \label{wolffine}
\int_{\R^n} (\ra\sigma_E)^{p'}dx \ge c\, \int_{\mathbb{R}^n} \wa\sigma_E \,d\sigma_E =c\int_{E} \wa\sigma_E\,d\sigma. 
\end{equation}
Hence,
\begin{equation} \label{wolff3} 
\int_{E} \wa\sigma_E\, d\sigma_E \le c\, \sigma(E), \quad \text { for all compact sets } E.
\end{equation}

If $0< s < 1$, then clearly the preceding estimate, together with  H\"older's inequality, yields \eqref{wolffintegral}.
If $s >1 $, we will show that \eqref{wolffintegral} holds by using a shifted dyadic lattice 
$\dy_t$ as in \cite{COV}. Let $E$ be a compact set in $\R^n$. Then we have the following estimate for the truncated Wolff potential $\mathbf{W}^r_{\alpha,p}\sigma_E$ (see \cite{COV}), 
$$\mathbf{W}^r_{\alpha,p}\sigma_E(x) \le c_1\,r^{-n} \int_{|t|\le cr}\sum_{Q \in \dy} \left[\frac{\sigma_E(Q+t)}{|Q+t|^{1-\frac{\alpha p}{n}}} \right]^{\frac{1}{p-1}}\chi_{Q+t}(x)dt.$$
Raising both sides to the power $s$, integrating over $E$ with respect to $d\sigma_E$, and using Minkowski's inequality, we obtain 
\begin{equation*}
\begin{aligned} 
& \int_{E}(\mathbf{W}^r_{\alpha,p}\sigma_E(x))^s \,d\sigma_E \\ 
& \le c_1\,\left(r^{-n} \int_{|t|\le cr} \left( \int_{E}\left[\sum_{Q \in \dy} \left[\frac{\sigma_E(Q+t)}{|Q+t|^{1-\frac{\alpha p}{n}}} \right]^{\frac{1}{p-1}}\chi_{Q+t}(x)\right]^s \, d\sigma_E \right)^{\frac{1}{s}}dt\right)^s.
  \end{aligned}
  \end{equation*}
Applying Proposition 2.2 in \cite{COV}, we have
\begin{equation*}
\begin{aligned} 
& \int_{E}\left[\sum_{Q \in \dy} \left[\frac{\sigma_E(Q+t)}{|Q+t|^{1-\frac{\alpha p}{n}}} \right]^{\frac{1}{p-1}}\chi_{Q+t}(x)\right]^s\,d\sigma_E \\
& \le c \, \sum_{Q \in \dy} \left[\frac{\sigma_E(Q+t)}{|Q+t|^{1-\frac{\alpha p}{n}}} \right]^{\frac{1}{p-1}}\sigma_E(Q+t)\\
& \times \left( \frac{1}{\sigma_E(Q+t)}\sum_{R\subset Q}
\left[\frac{\sigma_E(R+t)}{|R+t|^{1-\frac{\alpha p}{n}}} \right]^{\frac{1}{p-1}}\sigma_E(R+t)\right)^{s-1}.
 \end{aligned}
  \end{equation*}
Using Lemma 4.7 in \cite{JV1}, we see that the last factor is uniformly bounded, and hence by the same lemma 
$$\int_{E}\left[\sum_{Q \in \dy} \left[\frac{\sigma_E(Q+t)}{|Q+t|^{1-\frac{\alpha p}{n}}} \right]^{\frac{1}{p-1}}\chi_{Q+t}(x)\right]^s\,d\sigma_E  \le c \, \sigma(E).$$
Hence,
$$\int_{E}(\mathbf{W}^r_{\alpha,p}\sigma_E)^s \,d\sigma \le c_1\,\left(r^{-n} \int_{|t|\le cr} \left( c\,\sigma(E) \right)^{\frac{1}{s}}dt\right)^s \le c\, \sigma(E).$$  
Letting $r \to \infty$ and using the Monotone Convergence Theorem, we deduce 

$$\int_{E}(\wa\sigma_E)^s\,d\sigma \le c\, \sigma(E),$$ 
where the constant $c>0$ depends on $\al, p, n, s$ and $C$. 

Conversely, suppose that \eqref{wolffintegral} holds for a fixed $s>0$. Let $E$ be a compact set in $\R^n$ and suppose that $\sigma(E) > 0$. We write 
$$\sigma(E)=\int_E (\wa\sigma_E)^{-\beta} (\wa\sigma_E)^{\beta}d\sigma ,$$
where $\beta > 0$ will be chosen later. Using H\"older's inequality, we have
$$\sigma(E)\le \left(\int_E (\wa\sigma_E)^{-\beta r}d\sigma\right)^{\frac{1}{r}}\left(\int_E(\wa\sigma_E)^{\beta r'}d\sigma\right)^{\frac{1}{r'}},$$
where $ r >1$ and $r'=\frac{r}{r-1}$. Let us choose $\beta$ and $r$ such that $\beta r=p-1$ and  $\beta r'=s$; hence, $r=1+ \frac{p-1}{s} $ and $\beta=\frac{s(p-1)}{s+p-1}$. Consequently,
$$\sigma(E)\le \left(\int_E \frac{d\sigma_E}{(\wa\sigma_E)^{p-1}}\right)^{\frac{1}{r}}\left(\int_E(\wa\sigma_E)^s d\sigma\right)^{\frac{1}{r'}}.$$
Applying Theorem 1.11 in \cite{V1}, we have
$$\int_E \frac{d\sigma_E}{(\wa\sigma_E)^{p-1}} \le c\, \text{cap}_{\alpha,p}(E).$$
Combining this estimate and \eqref{wolffintegral} yields 
$$\sigma(E)\le c\left(\text{cap}_{\alpha,p}(E)\right)^{\frac{1}{r}}\left(\sigma(E)\right)^{\frac{1}{r'}}.$$
This proves 
$\sigma(E)\le c \, \text{cap}_{\alpha,p}(E)$.
Consequently, arguing as at the beginning of the proof, we also see that \eqref{wolffintegral} holds for any $s>0$. This completes the proof of  Lemma \ref{wolff-capacity}.
\end{proof}

\begin{rmk}
{\rm Suppose that \eqref{capcond} holds. Then from \eqref{wolffintegral} 
it follows  that, for any $s>0$, 
\begin{equation} \label{wolffball}
\int_Q (\wa\sigma_Q)^s d\sigma \le c(\al, p, n, C, s) \, \sigma(Q), 
\end{equation}
for all cubes $Q$ in $\R^n$.}

{\rm It can be shown using some results of the second author on duality 
of discrete Littlewood--Paley spaces 
(see, e.g., \cite{CO}, Lemma 4.5) that if \eqref{wolffball} holds for some  $s>0$, then it also holds for any $s>0$; consequently, \eqref{capcond} holds. 

The fact that  \eqref{wolffball} does not depend on $s$  was proved for $s>\frac{1}{2}$  by  Nazarov, Treil and Volberg \cite{NTV03}  using the Bellman function method.}
\end{rmk}

\section{Solutions of the nonlinear integral equation} \label{solnintegral}
\subsection{Inhomogeneous problems} 

Consider the equation
\begin{equation} \label{continteq}
u = \wa(u^qd\sigma )+ r, \quad u >0,
\end{equation}
where $r >0$, and $u \in L^{q}_{\rm loc}(\R^n, d\sigma)$. We recall that if $ u \ge \wa (u^qd\sigma)$, then
\begin{equation} \label{lowerest}
u(x) \ge c\, \left(\wa\sigma(x)\right)^{\frac{p-1}{p-1-q}}, \quad x \in \R^n,
\end{equation}
where $c=c(n,p,q,\al)>0$ (Theorem 3.4 in \cite{CV2}). Therefore, a necessary condition for the existence of a nontrivial solution to \eqref{continteq} is that $\wa\sigma \not \equiv \infty$, or equivalently,
\begin{equation} \label{Wolff-finite}
\int_1^{\infty} \left(\frac{\sigma(B(0,t))}{t^{n- \al p}}\right)^{\frac{1}{p-1}}\frac{dt}{t} < \infty.
\end{equation} 

\begin{theorem} \label{inhomocase}
Let $r >0, 1<p<n, 0 < q <p-1$, and $0 < \al < \frac{n}{p}$. Let  $\sigma \in M^+(\R^n)$. Suppose that both \eqref{capcond}  and \eqref{Wolff-finite}  hold. Then there exists a solution $u$ to  \eqref{continteq} such that  
\begin{equation} \label{match-est}
c^{-1}\,\left(r + (\wa\sigma)^{\frac{p-1}{p-1-q}}\right) \le u \le c \, \left(r+(\wa\sigma)^{\frac{p-1}{p-1-q}}\right),
\end{equation}
where $c=c(n,p,q,\alpha,C,r)$. Moreover, $u \in L^{s}_{\rm loc}(\R^n, d\sigma)$, for every $s>0$.
\end{theorem}
\begin{proof}
Let $w=c_0\Big(r+ \left(\wa\sigma\right)^{\frac{(p-1)}{p-1-q}}\Big )$, where $c_0>0$ is a sufficiently small constant. Using Lemma 3.2 in \cite{CV1}, we see that $w$ is a subsolution to \eqref{continteq}.
Let $$v=c\,\left(r+ \left(\wa\sigma\right)^{\frac{(p-1)}{p-1-q}}\right),$$
where $c>0$ is a large constant to be determined later. We will show that $v$ is a supersolution of \eqref{continteq}. First, we estimate
\begin{equation*}
\begin{aligned} 
& \wa(v^qd\sigma)(x) =\int_0^{\infty} \left( \frac{\int_{B(x,t)}v^qd\sigma}{t^{n-\al p}}\right)^{\frac{1}{p-1}}\frac{dt}{t}\\
& \le c^{\frac{q}{p-1}}c_1\left(r^{\frac{q}{p-1}}\,\wa\sigma+ \int_0^{\infty} \left(\frac{\int_{B(x,t)}(\wa\sigma)^{\frac{(p-1)q}{p-1-q}}d\sigma}{t^{n-\al p}}\right)^{\frac{1}{p-1}}\frac{dt}{t}\right).
\end{aligned}
\end{equation*} 
Letting $\beta=\frac{(p-1)q}{p-1-q}$, we next estimate 
\begin{equation*}
\begin{aligned} 
& \int_{B(x,t)}(\wa\sigma)^{\beta}d\sigma  =\int_{B(x,t)}\left[\int_0^{\infty} \left(\frac{ \sigma(B(y,s))}{s^{n-\al p}}\right)^{\frac{1}{p-1}}\frac{ds}{s}\right]^{\beta}d\sigma(y)\\
& \le c_1\,\int_{B(x,t)}\left[\int_0^{t} \left(\frac{ \sigma(B(y,s))}{s^{n-\al p}}\right)^{\frac{1}{p-1}}\frac{ds}{s}\right]^{\beta}d\sigma(y)\\
& +c_1\,\int_{B(x,t)}\left[\int_t^{\infty} \left(\frac{ \sigma(B(y,s))}{s^{n-\al p}}\right)^{\frac{1}{p-1}}\frac{ds}{s}\right]^{\beta}d\sigma(y)=c_1\,(I + II).
\end{aligned}
\end{equation*}

For $y \in B(x,t)$ and $s \ge t$, we have $ B(y,s) \subset B(x,2s)$. Hence, 
\begin{equation*}
\begin{aligned} 
II & \le \int_{B(x,t)}\left[\int_t^{\infty} \left(\frac{ \sigma(B(x,2s))}{s^{n-\al p}}\right)^{\frac{1}{p-1}}\frac{ds}{s}\right]^{\beta}d\sigma(y)\\
& = \sigma(B(x,t))\left[\int_t^{\infty} \left(\frac{ \sigma(B(x,2s))}{s^{n-\al p}}\right)^{\frac{1}{p-1}}\frac{ds}{s}\right]^{\beta}\\ &\le c_1\, \sigma(B(x,t))[\wa\sigma(x)]^{\beta},
\end{aligned}
\end{equation*} 
where $c_1=c_1(\al, n, p, q)$. For $s\le t$, we have $B(y,s) \subset B(x,2t)$, so
\begin{equation*}
\begin{aligned} 
I &=\int_{B(x,t)}\left[\int_0^{t} \left(\frac{ \sigma(B(y,s)\cap B(x,2t))}{s^{n-\al p}}\right)^{\frac{1}{p-1}}\frac{ds}{s}\right]^{\beta}d\sigma(y)\\
& \le \int_{B(x,2t)}\left[\int_0^{t} \left(\frac{ \sigma(B(y,s)\cap B(x,2t))}{s^{n-\al p}}\right)^{\frac{1}{p-1}}\frac{ds}{s}\right]^{\beta}d\sigma(y)\\
&\le \int_{B(x,2t)}\left[\wa\sigma_{B(x,2t)}\right]^{\beta}d\sigma.
\end{aligned}
\end{equation*} 

Using \eqref{wolffball}, we obtain 
$$\int_{B(x,2t)}\left[\wa\sigma_{B(x,2t)}\right]^{\beta}d\sigma\le c_2 \,\sigma(B(x,2t),$$
where $c_2=c_2(\al, n, p, q, C)$. Thus,
\begin{equation*} 
\int_{B(x,t)}(\wa\sigma)^{\beta}d\sigma \le c_1\,[\wa\sigma(x)]^{\beta} \sigma(B(x,t) +c_2\,\sigma(B(x,2t)).
\end{equation*}
Consequently,
\begin{equation*}
\begin{aligned} 
\left[\int_{B(x,t)}(\wa\sigma)^{\beta}d\sigma\right]^{\frac{1}{p-1}}
& \le c_1\,(\wa\sigma(x))^{\frac{q}{p-1-q}}[\sigma(B(x,t))]^{\frac{1}{p-1}}\\ & +c_2\,[\sigma(B(x,2t))]^{\frac{1}{p-1}}. 
\end{aligned}
\end{equation*} 
Hence, 
\begin{equation*}
\begin{aligned} 
\wa(v^qd\sigma) & 
\le c^{\frac{q}{p-1}}\tilde{c} \left( r^{\frac{q}{p-1}}\,\wa\sigma 
+  c_2\,\wa\sigma + c_1\,\left(\wa\sigma \right)^{\frac{p-1}{p-1-q}}\right)\\
& \le c^{\frac{q}{p-1}}c_3\left(r + \left(\wa\sigma \right)^{\frac{p-1}{p-1-q}}\right),
\end{aligned}
\end{equation*} 
where $c_3=c_3(n,p,q,\al,C,r)$. Therefore, picking $c$ large enough yields $v \ge r+\wa(v^qd\sigma)$ and $ v \ge w$.

Given a supersolution $v$, and a subsolution $w$ such that $w \le v$, we use a standard iteration argument,  and the Monotone Convergence Theorem, to deduce the existence of a nontrivial solution $u$ to \eqref{continteq} such that  $w \le u\le v$, which yields \eqref{match-est}. By Lemma \ref{wolff-capacity}, we deduce that $u \in L^{s}_{\text{loc}}(\R^n, d\sigma)$ for every $s>0$. This concludes the proof of Theorem \ref{inhomocase}.
\end{proof}

\subsection{Homogeneous problems} 
Having the same hypotheses as in Theorem \ref{inhomocase}, we obtain the following result for the homogeneous equation 
\begin{equation}\label{homoeq}
u=\wa(u^qd\sigma), \quad u \ge 0, 
\end{equation}
where $u \in L^{q}_{\rm loc}(\R^n, d\sigma)$.

\begin{theorem} \label{homocase}
Let $1<p<\infty, 0 < q <p-1$, and $0 < \al < \frac{n}{p}$. Let $\sigma \in M^+(\R^n)$. Suppose that both \eqref{capcond} and \eqref{Wolff-finite} hold. Then there exists a solution $u$ to  \eqref{homoeq} such that 
\begin{equation} \label{twosided-est}
c^{-1} \left(\wa\sigma\right)^{\frac{p-1}{p-1-q}} \le u \le c \, \left(\wa\sigma+(\wa\sigma )^{\frac{p-1}{p-1-q}}\right),
\end{equation}
where $c=c(n,p,q,\alpha,C)$. Moreover, $u \in L^{s}_{\rm loc}(\R^n, d\sigma)$, for every $s>0$.
\end{theorem}
\begin{proof}

Let $w= c_0 \left(\wa\sigma\right)^{\frac{p-1}{p-1-q}}$ with a sufficiently small constant $c_0>0$. Then as in the proof of Theorem~\ref{inhomocase},  $w$ is a subsolution to \eqref{homoeq}.
Now, let 
$$v= c_1\left(\wa\sigma + \left(\wa\sigma\right)^{\frac{p-1}{p-1-q}}\right),$$
where $c_1$ is a large constant. Clearly,  $ w \le v$. Arguing as in in the proof of Theorem \ref{inhomocase} and using \eqref{wolffball}, we show that $v$ is a supersolution to \eqref{homoeq}. 
Therefore, as above, we see that there  exists a nontrivial solution $u$ to \eqref{homoeq} such that \eqref{twosided-est} holds.

Given $s>0$, then  $u \in L^{s}_{\text{loc}}(\R^n,d\sigma)$ follows from \eqref{twosided-est} and Lemma \ref{wolff-capacity}. 
We also notice that $\liminf_{|x| \to \infty }u(x)=0$ since $\liminf_{|x| \to \infty } \wa\sigma(x)=0$ by Corollary 3.2 in \cite{CV2}. This completes the proof of Theorem \ref{homocase}. 
\end{proof}

Instead of  the capacity condition \eqref{capcond}, let us use now a weaker  condition 
\begin{equation} \label{new}
 \wa\left(\left( \wa\sigma\right )^{\frac{(p-1)q}{p-1-q}}d\sigma\right) \le \kappa \left( \wa\sigma  + \left( \wa\sigma\right )^{\frac{p-1}{p-1-q}}\right) < \infty \, \text{ a.e.},
\end{equation}
where $\kappa $ is a positive constant. Then we can still construct solutions to \eqref{homoeq} satisfying the Brezis--Kamin type estimates, and show  that \eqref{new} is actually necessary for the existence 
of such solutions. 
\begin{theorem} \label{mainthm}
Let $1<p<\infty, 0 < q <p-1$, and $0 < \al < \frac{n}{p}$. If $\sigma \in M^+(\R^n)$ satisfies  \eqref{new}, then there exists a solution $u$ to \eqref{homoeq} such that \eqref{twosided-est} holds 
with a constant  $c>0$ depending only on $\al, n, p, q,\kappa$. 

Conversely, suppose that there exists a nontrivial supersolution $u$ to \eqref{homoeq} such that \eqref{twosided-est} holds. Then \eqref{new} holds with $\kappa=\kappa(p,q,c)$.
\end{theorem}
\begin{proof}
Suppose that \eqref{new} holds. Let $w=c_0(\wa\sigma)^{\frac{p-1}{p-1-q}}$ with small constant  $c_0$; then $w$ is a subsolution to \eqref{homoeq} as before. Let
$$v =c\left( \wa\sigma  + \left( \wa\sigma\right )^{\frac{p-1}{p-1-q}}\right),$$
where $c>0$ is a large constant. We estimate
\begin{equation*}
\begin{aligned} 
& \wa(v^q d\sigma)  = c^{\frac{q}{p-1}}\wa\left( \left( \wa\sigma  + \left( \wa\sigma\right )^{\frac{p-1}{p-1-q}}\right)^q d\sigma \right)\\
& \le  a\,c^{\frac{q}{p-1}} \wa\left((\wa\sigma)^q d \sigma\right) +  a\,c^{\frac{q}{p-1}} \wa\left((\wa\sigma)^{\frac{(p-1)q}{p-1-q}} d \sigma\right)\\
& \le  a\,c^{\frac{q}{p-1}} \wa\left((\wa\sigma)^q d \sigma\right) +  a\,c^{\frac{q}{p-1}} \kappa \left( \wa\sigma + \left(\wa\sigma\right)^{\frac{p-1}{p-1-q}} \right),
\end{aligned}
\end{equation*} 
where $a=a(p,q)$. Next, we write
$$\wa\left((\wa\sigma)^q d \sigma\right)(x) = \int_0^{\infty} \left( \frac{\int_{B(x,t)}(\wa\sigma)^q d \sigma}{t^{n-\al p}}\right)^{\frac{1}{p-1}}\frac{dt}{t}.$$
Using H\"{o}lder's inequality and Young's inequality, we obtain 
\begin{equation*}
\begin{aligned} 
 \int_{B(x,t)}(\wa\sigma)^q d \sigma & \le \left(\int_{B(x,t)}(\wa\sigma)^{\frac{(p-1)q}{p-1-q}} d \sigma\right)^{\frac{p-1-q}{p-1}} \left[\sigma(B(x,t))\right]^{\frac{q}{p-1}} \\
& \le \tilde{b}\left(\int_{B(x,t)}(\wa\sigma)^{\frac{(p-1)q}{p-1-q}} d \sigma +   \sigma(B(x,t)) \right).
\end{aligned}
\end{equation*} 
Hence, 
\begin{equation*}
\begin{aligned} 
& \int_0^{\infty} \left( \frac{\int_{B(x,t)}(\wa\sigma)^q d \sigma}{t^{n- \al p}}\right)^{\frac{1}{p-1}}\frac{dt}{t}\\ & \le 
b\int_0^{\infty} \left( \frac{\int_{B(x,t)}(\wa\sigma)^{\frac{(p-1)q}{p-1-q}} d \sigma}{t^{n- \al p}}\right)^{\frac{1}{p-1}}\frac{dt}{t}  + b\int_0^{\infty} \left( \frac{\sigma(B(x,t))}{t^{n-\al p}}\right)^{\frac{1}{p-1}}\frac{dt}{t}\\
& = b\wa\left((\wa\sigma)^{\frac{(p-1)q}{p-1-q}} d \sigma\right)(x) + b \wa\sigma(x), 
\end{aligned}
\end{equation*} 
where $b=b(p,q)$. 
By \eqref{newcond}, the last term is bounded by 
$$b\kappa \left( \wa\sigma + \left(\wa\sigma\right)^{\frac{p-1}{p-1-q}} \right)(x) +  b\wa\sigma(x)$$
$$=  b( \kappa+1) \wa\sigma + b\kappa \left(\wa\sigma\right)^{\frac{p-1}{p-1-q}}.$$
It follows, 
\begin{equation*}
\begin{aligned} 
\wa(v^qd\sigma) & \le a\,c^{\frac{q}{p-1}}b( \kappa+1) \wa\sigma + a\,c^{\frac{q}{p-1}}b\kappa \left(\wa\sigma\right)^{\frac{p-1}{p-1-q}}\\
& + a\,c^{\frac{q}{p-1}} \kappa \left( \wa\sigma + \left(\wa\sigma\right)^{\frac{p-1}{p-1-q}} \right)\\ & \le \,c^{\frac{q}{p-1}}a(b\kappa + b + \kappa)( \wa\sigma + \left(\wa\sigma\right)^{\frac{p-1}{p-1-q}}).
\end{aligned}
\end{equation*} 
If $c$ is chosen so that $c\ge\,c^{\frac{q}{p-1}}a(b\kappa + b + \kappa)$ and $c \ge c_0$, then we obtain $v \ge w$ and $v \ge \wa(v^qd\sigma)$. Using  iterations as above, and the Monotone Convergence Theorem, we deduce 
from this the existence of a solution $u$ to the equation $ u = \wa(u^qd\sigma)$ which satisfies \eqref{twosided-est}. 

Conversely, suppose that there exists a nontrivial supersolution $u$ to \eqref{homoeq}  such that \eqref{twosided-est} holds. Then clearly  by the lower bound in \eqref{twosided-est}, 
  $$u \ge \wa(u^q d\sigma) \ge (c^{-1})^{\frac{q}{p-1}} \wa\left((\wa\sigma)^{\frac{(p-1)q}{p-1-q}} d\sigma\right) .$$
Using now 
the upper estimate in \eqref{twosided-est}, we deduce \eqref{new} with $\kappa=\kappa(p,q,c)$.  
This completes the proof of Theorem  \ref{mainthm}.
\end{proof}

\section{Proofs of Theorem \ref{continhomothm}, Theorem \ref{conthomothm}, and Theorem \ref{newthm} } \label{proofthm}

\begin{proof}[Proof of Theorem \ref{continhomothm}]
Let $v \in L^{1+q}_{\text{loc}}(\R^n,d\sigma) $ be a solution to the integral equation  
\begin{equation} \label{wolfintine}
v = K\w(v^qd\sigma )+r,
\end{equation}
where $K$ is the constant in Corollary 4.5 in \cite{PV1} (see also \cite{KM}). 
Then $\liminf_{|x| \to \infty }v(x)=r$, and $v$ satisfies 
\begin{equation} \label{Vestimate1}
c^{-1}\,\left(r + (\w\sigma)^{\frac{p-1}{p-1-q}}\right) \le v \le c \, \left(r+(\w\sigma)^{\frac{p-1}{p-1-q}}\right).
\end{equation}
The existence of such a $v$ follows from Theorem \ref{inhomocase} with some modifications in the constants. We have
$$
\int_B \w (v^q d \sigma_B) \, v^q \, d \sigma \le c \int_B v^{1+q} d \sigma < \infty,
$$ 
for every ball $B$. By a local version of Wolff's inequality (see \cite{AH}, Theorem 4.5.5), we see that $v^q d \sigma \in W^{-1, p'}_{{\rm loc}} (\R^n)$.

We set $u_0=r$ and $B_k=B(0,2^k)$, where $k=0,1,2 \ldots $. We have $ u^q_0 d\sigma \in W^{-1,p'}(B_k)$ since $u_0 \le v$ and $v^q d \sigma \in W^{-1, p'}_{{\rm loc}} (\R^n)$. Hence, there exists a unique $p$-superharmonic solution $u^k_1$ to the equation 
\begin{equation}\label{u1k-seq}
-\Delta_pu^k_1=\sigma u^q_0 \text{ in } B_k \, , \quad u_1^k \ge r, \, u_1^k-r \in W_0^{1,p}(B_k).
\end{equation}
(See, e.g., Theorem 21.6 in \cite{HKM}.) By Corollary 4.5 in \cite{PV1}, we have 
$$u^{k}_1 -r \le K \w(u^q_0d\sigma).$$
Since $u_0 \le v$, we get
$$u^{k}_1 \le K \w(v^qd\sigma)+r = v.$$
We see that the sequence $\{u_1^k\}_k$ is increasing by a comparison principle (Lemma 5.1 in \cite{CV2}). Letting $u_1= \lim_{ k \to \infty}u_1^k$ and using the weak continuity of the $p$-Laplacian (\cite{TW1}) and the Monotone Convergence Theorem, we deduce that $u_1$ is a $p$-superharmonic solution to the equation
$$-\Delta_pu_1=\sigma u^q_0 \text{ in } \R^n.$$
Moreover, $r \le u_1 \le v$ since $u_1^k \le v$ and hence $\liminf_{|x| \to \infty }u_1(x)=r$. Clearly, we also have $u_0 \le u_1$.
We notice that $u^q_0 d\sigma \in  W^{-1, p'}_{{\rm loc}} (\R^n)$ since $u_0 \le v$ and $ v^q d\sigma \in  W^{-1, p'}_{{\rm loc}} (\R^n)$. Therefore, applying Lemma 3.3 in \cite{CV2}, we conclude that $u_1 \in W_{\text{loc}}^{1,p}(\R^n)$.

Let us now construct by induction a sequence $\{u_j\}_j$ of $p$-superharmonic functions in $\R^n$, $u_j \in L^q_{\text{loc}}(\R^n,d\sigma)$, so that  
\begin{equation}\label{entire1bvapprox}
\left\{ \begin{array}{ll}
-\Delta_pu_{j}=\sigma u_{j-1}^q \text{ in ~~} \mathbb{R}^n, \quad j=2,3,\ldots ,\\
r\le u_{j} \le v,\, u_j \in W_{\text{loc}}^{1,p}(\R^n), \\
u_{j-1} \le u_j  , \\
\liminf_{|x| \to \infty }u_j(x)=r.
\end{array} \right.
\end{equation}
Suppose that $u_1, \ldots , u_{j-1}$ have been constructed. 
We see that $\sigma u^q_{j-1} \in W^{-1,p'}(B_k)$ since $u_{j-1} \le v$ and $v^q d \sigma \in W^{-1, p'}_{{\rm loc}} (\R^n)$. Thus, as before, there exists a unique $p$-superharmonic solution $u^k_j$ to the equation 
\begin{equation}\label{uk}
-\Delta_pu^k_j=\sigma u^q_{j-1} \text{ in } B_k, \quad u_j^k \ge r, \, u_j^k-r \in W_0^{1,p}(B_k).
\end{equation}
Arguing by induction, let $u_{j-1}^k$ be the unique solution of the equation
\begin{equation}\label{uk1}
-\Delta_pu^k_{j-1}=\sigma u^q_{j-2} \text{ in } B_k, \quad u_{j-1}^k \ge r, \, u_{j-1}^k-r \in W_0^{1,p}(B_k).
\end{equation}
Since $u_{j-2} \le u_{j-1} $, by the comparison principle we deduce that
\begin{equation} \label{wcp1}
u^k_{j-1} \le u^k_j, \quad \text{for all} \,\,  k \ge 1.
\end{equation}
Using Corollary 4.5 in \cite{PV1}, we have 
$$u^{k}_j -r \le K \w(u^q_{j-1}d\sigma).$$
Since $u_{j-1} \le v$, we obtain
$$u^{k}_j \le K \w(v^qd\sigma)+r = v.$$

Applying again the comparison principle, we see that the sequence $\{u_j^k\}_k$ is increasing. Hence, letting $u_j= \lim_{ k \to \infty}u_j^k$ and using the weak continuity (\cite{TW1})  and the Monotone Convergence Theorem, we see that $u_j$ is a $p$-superharmonic solution to the equation
$$-\Delta_pu_j=\sigma u^q_{j-1} \text{ in } \R^n.$$
Moreover, $r \le u_j \le v$ since $u_j^k \le v$ and hence $\liminf_{|x| \to \infty }u_j(x)=r$. We also have $u_{j-1} \le u_j$ since $u^k_{j-1} \le u^k_j$. We see that $u^q_{j-1}d\sigma \in  W^{-1, p'}_{{\rm loc}} (\R^n)$ since $u_{j-1} \le v$ and $ v^q d\sigma \in  W^{-1, p'}_{{\rm loc}} (\R^n)$. Therefore, by Lemma 3.3 in \cite{CV2}, it follows that $u_j \in W_{\text{loc}}^{1,p}(\R^n)$. By the weak continuity (\cite{TW1})  and the Monotone Convergence Theorem, we deduce that $u$ is a solution to the equation 
$$-\Delta_pu=\sigma u^q \text{ in } \R^n.$$
Furthermore, $r \le u \le v,$ and hence $\liminf_{ |x| \to \infty}u(x)=r $. By \eqref{Vestimate1}, we get
$$u \le c\Bigl( r +\w{\sigma}\Bigr)^{\frac{p-1}{p-1-q}}.$$ 
Using Lemma 3.3 in \cite{CV2} again, we conclude that $u \in W_{\text{loc}}^{1,p}(\R^n)$ since $ u^q d\sigma \in  W^{-1, p'}_{{\rm loc}} (\R^n)$. The lower estimate follows from \eqref{lowerest} and the fact that $u \ge r$. This completes the proof of  Theorem \ref{continhomothm}.  
\end{proof}

\begin{proof}[Proof of Theorem \ref{conthomothm}]
Suppose that \eqref{Wolfinite} and \eqref{capacitycond} hold. Then by Theorem \ref{homocase} there exists a nontrivial solution  $v \in L^{1+q}_{\text{loc}}(\R^n,d\sigma)$ to the equation
\begin{equation}\label{V2}
v = K \w(v^qd\sigma), \quad \liminf_{|x| \to \infty }v(x)=0. 
\end{equation}
Moreover, there exists a constant $c=c(n,p,q,C(\sigma))>0$ such that
\begin{equation} \label{Vupper}
v \le c\left(\w\sigma+\left(\w\sigma\right)^{\frac{p-1}{p-1-q}}\right).
\end{equation}
Arguing as in the proof of Theorem 1.1 in \cite{CV2}, we deduce the existence of a minimal $p$-superharmonic solution $u$ to the equation 
$$-\Delta_pu=\sigma u^q \text{ in } \R^n.$$

Since $u \le v$, we see that $u \in L_{\text{loc}}^{1+q}(\R^n, d\sigma)$. Hence,
$$
\int_B \w (u^q d \sigma_B) \, u^q \, d \sigma \le c\, \int_B u^{1+q} d \sigma < \infty,
$$ 
for every ball $B$. By the local  Wolff's inequality, it follows  that $u^q d \sigma \in W^{-1, p'}_{{\rm loc}} (\R^n)$. 
Hence $u \in W^{1, p}_{{\rm loc}} (\R^n)$ by Lemma 3.3 in \cite{CV2}. 

Moreover, by \eqref{lowerest}, $ u \ge  c\, \left(\w\sigma\right)^{\frac{p-1}{p-1-q}}$, and consequently, 
$$c^{-1}\left(\w\sigma\right)^{\frac{p-1}{p-1-q}} \le u \le  c\left(\w\sigma+\left(\w\sigma\right)^{\frac{p-1}{p-1-q}}\right).$$
The case $p\ge n$ follows from Theorem 2.4 (ii) in \cite{CV2}. This completes the proof of  Theorem \ref{conthomothm}.
\end{proof}

\begin{proof}[Proof of Theorem \ref{newthm}] Suppose there exists a $p$-superharmonic solution $u$ to \eqref{eq1} with $r=0$, and $u$ satisfies \eqref{two-sided}. By Corollary 4.5 in \cite{PV1}, $u \ge \frac{1}{K} \w(u^q d\sigma)$. Consequently, by \eqref{lowerest}, $u \ge C (\w\sigma)^{\frac{p-1}{p-1-q}}$; and hence, $$u \ge c \w((\w\sigma)^{\frac{(p-1)q}{p-1-q}}d\sigma).$$
Therefore,
$$\w((\w\sigma)^{\frac{(p-1)q}{p-1-q}}d\sigma) \le   c\left(\w\sigma+\left(\w\sigma\right)^{\frac{p-1}{p-1-q}}\right) < \infty \, \text{  a.e.}$$

Conversely, suppose that \eqref{newcond} holds, then applying Theorem \ref{mainthm} and arguing as in the proof of Theorem 1.1 in \cite{CV2}, we conclude the proof of  Theorem \ref{newthm}.
\end{proof}

\begin{rmk}
{\rm Since our approach is based on  the Wolff potential estimates (see \cite{KM}, \cite{KuMi}, \cite{Lab}, \cite{TW1}, \cite{PV1}), 
all of the results metioned above remain 
valid  if one replaces the $p$-Laplacian $\Delta_p$ in the model problem \eqref{eq1} by a more general quasilinear operator $ \textrm{div} \mathcal{A}(x,\nabla \cdot)$, under standard structural assumptions on $\mathcal{A}(x,\xi)$ which ensure that $\mathcal{A}(x,\xi) \cdot \xi \approx  |\xi|^p$, or a fully nonlinear operator of $k$-Hessian type (see details in \cite{CV2})}.
\end{rmk}

\section{The radial case } \label{radial}

In this section we will assume that $\sigma\in M^+(\R^n)$ is radially symmetric. 
Suppose $0<q<1$ and $0<2 \alpha< n$. We study equation \eqref{frac} with $r=0$, i.e.,
\begin{equation}\label{subeq}
\left\{ \begin{array}{ll}
(-\Delta)^{\al} u = \sigma u^q \quad \text{ in } \R^n,\\
\liminf_{ |x| \to \infty}u(x)=0.
\end{array} \right.
\end{equation}
We notice that a necessary condition for the existence of a nontrivial solution to \eqref{subeq} is that $\sigma$ must be absolutely continuous with respect to the $(\al,2)$-capacity $\text{cap}_{\al,2}(\cdot)$ \cite{CV2}. In particular,  $\sigma$ has no atoms. Suppose that $\sigma$ is radial and $\rat\sigma \not \equiv + \infty$, then we have (\cite{Rub}, p. 231)
$$\rat\sigma(x)=c\,\int_{\R^n} \int_{\max\{|x|,|y|\}}^{\infty}(|t|^2-|x|^2)^{\al-1}(|t|^2-|y|^2)^{\al-1}\frac{dt}{t^{2\al+n-3}}\,d\sigma(y)  $$
$$\ap  \frac{\sigma(B(0,|x|)}{|x|^{n-2\al}} + \int_{|y|\ge|x|} \frac{d\sigma(y)}{|y|^{n-2\al} } ,$$
where we drop the first term if $x=0$. We have the following theorem.
\begin{theorem} \label{radpro} Let $0<q<1, n \ge 1$, and $0<\al<\frac{n}{2}$.   Let $\sigma \in M^+(\R^n)$ be radially symmetric. Then there exists a nontrivial solution $u$ to \eqref{subeq} if and only if \eqref{radialcond} holds. Moreover, $u$ satisfies \eqref{radlow}.
\end{theorem}

\begin{proof}[Proof of Theorem \ref{radpro}]
We remark that \eqref{subeq} is understood in the sense that $u=\rat(u^qd\sigma)$. Suppose that $u$ is a solution to \eqref{subeq}. We first notice that $\rat\sigma$ is radial. Therefore, the minimal solution to \eqref{subeq} constructed in Theorem 4.8 in \cite{CV2} is radial as well. Hence, we may assume that $u$ is radial. By \eqref{lowerest},
$$ u(x) \ge c \, \left(\rat\sigma(x)\right)^{\frac{1}{1-q}}, \quad x \in \R^n,$$
where $c=c(n,q)>0$. Consequently,
$$ u(x) \ge c \, \left(\int_{|y|\ge|x|} \frac{d\sigma(y)}{|y|^{n-2\al}}\right)^{\frac{1}{1-q}}, \quad x \in \R^n.$$
We have
\begin{equation} \label{radeq1}
u(x) = \rat(u^qd\sigma)(x) \approx \frac{\int_{|y| < |x|}u^qd\sigma}{|x|^{n-2\al}} + \int_{|y|\ge|x|} \frac{u^qd\sigma(y)}{|y|^{n-2\al} }. 
\end{equation}
By Lemma 4.2 in \cite{CV2}, we have,  for all $\nu \in M^+(\R^n)$,  
$$ ||\rat\nu ||_{L^q(d\sigma_{B(0,|x|)})}  \le c\,\left(\int_{B(0,|x|)}u^qd\sigma\right)^{1-q} \, \nu(\R^n).$$ 
Let $\nu=\delta_0$, we get
$$\left(\int_{B(0,|x|)}\frac{d\sigma(y)}{|y|^{(n-2\al)q}}\right)^{\frac{1}{1-q}} \le c\,\int_{B(0,|x|)}u^qd\sigma.$$
Therefore, we deduce from \eqref{radeq1} 
$$u(x) \ge c\,\frac{1}{|x|^{n-2\al}}\left(\int_{B(0,|x|)}\frac{d\sigma(y)}{|y|^{(n-2\al)q}}\right)^{\frac{1}{1-q}}.$$
 Thus,  \eqref{radialcond} follows since $u\not \equiv + \infty$.

Conversely, suppose that condition \eqref{radialcond} holds. This implies 
$$\int_{B(0,|x|)} \frac{d\sigma(y)}{|y|^{(n-2\al)q}} < \infty \text { and } \int_{|y|\ge |x|} \frac{d\sigma(y)}{|y|^{n-2\al}}  < \infty,\quad x \neq 0.$$ Let 
$u_0 = c_0\,\left(\rat\sigma \right)^{\frac{1}{1-q}}$ with a small constant $c_0$. Then $u_0 \le \rat(u_0^qd\sigma)$ as before. 
$$\text{Let } v(x)= c \, \left (\frac{1}{
|x|^{n-2\al}}\left(\int_{|y|<|x|} \frac{d\sigma(y)}{|y|^{(n-2\al)q}}\right)^{\frac{1}{1-q}} + \left(\int_{|y|\ge|x|} \frac{d\sigma(y)}{|y|^{n-2\al}}\right)^{\frac{1}{1-q}} \right ).$$
We see that $v \ge \rat(v^qd\sigma)$. Indeed, for $x \neq 0$,  we have 
\begin{equation*}
\begin{aligned} 
\rat(v^qd\sigma)(x) & \ap \frac{1}{|x|^{n-2\al}}\int_{|y|<|x|}v^qd\sigma+\,\int_{|y| \ge|x|}\frac{v^qd\sigma}{|y|^{n-2\al}}\\
& \le c^q\,\frac{1}{|x|^{n-2\al}}\int_{|y|<|x|}\frac{1}{|y|^{(n-2\al)q}}\left(\int_{|z|<|y|} \frac{d\sigma(z)}{|z|^{(n-2\al)q}}\right)^{\frac{q}{1-q}}d\sigma(y)\\
&+ c^q\,\frac{1}{|x|^{n-2\al}}\int_{|y|<|x|}\left(\int_{|z|\ge |y|} \frac{d\sigma(z)}{|z|^{n-2\al}}\right)^{\frac{q}{1-q}}d\sigma(y)\\
&+c^q\,\int_{|y| \ge|x|}\frac{1}{|y|^{n-2\al}}\frac{1}{|y|^{(n-2\al)q}}\left(\int_{|z|<|y|} \frac{d\sigma(z)}{|z|^{(n-2\al)q}}\right)^{\frac{q}{1-q}}d\sigma(y)\\
&+c^q\,\int_{|y| \ge|x|}\frac{1}{|y|^{n-2\al}}\left(\int_{|z|\ge |y|} \frac{d\sigma(z)}{|z|^{n-2\al}}\right)^{\frac{q}{1-q}}d\sigma(y)\\&:=c^q(I+II+III+IV).
\end{aligned}
\end{equation*} 
Clearly,
\begin{equation*} 
I\le  \frac{1}{|x|^{n-2\al}}\left(\int_{|y|<|x|} 
\frac{d\sigma(y)}{
|y|^{(n-2\al)q}}\right)^{\frac{1}{1-q}}.
\end{equation*}
We next estimate
\begin{eqnarray*}
II & = & \frac{1}{|x|^{n-2\al}}\int_{|y|<|x|} \left(\int_{|y|\le|z|<|x|} \frac{d\sigma(z)}{|z|^{n-2\al}}+\int_{|z|\ge|x|} \frac{d\sigma(z)}{|z|^{n-2\al}}\right)^{\frac{q}{1-q}}d\sigma(y)\\
& \le & c_1\frac{1}{|x|^{n-2\al}}\int_{|y|<|x|} \left(\int_{|y|\le|z|<|x|} \frac{d\sigma(z)}{|z|^{n-2\al}}\right)^{\frac{q}{1-q}}d\sigma(y)\\ 
& + & c_1\,\frac{1}{|x|^{n-2\al}}\int_{|y|<|x|} d\sigma(y)\left(\int_{|z|\ge|x|} \frac{d\sigma(z)}{|z|^{n-2\al}}\right)^{\frac{q}{1-q}} \\ 
&= & c_1(II_a+II_b).
\end{eqnarray*}
Next,
\begin{equation*}
\begin{aligned} 
II_a & =\frac{1}{|x|^{n-2\al}}\int_{|y|<|x|} \left(\int_{|y|\le|z|<|x|} \frac{d\sigma(z)}{|z|^{(n-2\al)q}|z|^{(n-2\al)(1-q)}}\right)^{\frac{q}{1-q}}d\sigma(y)\\
& \le \frac{1}{|x|^{n-2\al}}\int_{|y|<|x|} \left(\int_{|y|\le|z|<|x|} \frac{d\sigma(z)}{|z|^{(n-2\al)q}}\right)^{\frac{q}{1-q}}\frac{1}{|y|^{(n-2\al)q}} d\sigma(y)\\
& \le \frac{1}{|x|^{n-2\al}} \left(\int_{|y|<|x|} \frac{d\sigma(y)}{|y|^{(n-2\al)q}}\right)^{\frac{1}{1-q}}.
\end{aligned}
\end{equation*} 
Using Young's inequality with exponents $\frac{1}{1-q}$ and $\frac{1}{q}$, we obtain
$$II_b \le c_2\, \left(\left(\frac{1}{|x|^{n-2\al}}\int_{|y|<|x|} d\sigma(y)\right)^{\frac{1}{1-q}}+ \left(\int_{|z|\ge|x|} \frac{d\sigma(z)}{|z|^{n-2\al}}\right)^{\frac{1}{1-q}} \right).$$
We next estimate  
\begin{equation*}
\begin{aligned} 
&III \le c_1\,\int_{|y|\ge|x|} \frac{1}{|y|^{n-2\al}|y|^{(n-2\al)q}}\left(\int_{|z|<|x|}\frac{d\sigma(z)}{|z|^{(n-2\al)q}}\right)^{\frac{q}{1-q}}d\sigma(y)\\
&+c_1\,\int_{|y|\ge|x|} \frac{1}{|y|^{n-2\al}|y|^{(n-2\al)q}}\left(\int_{|x|\le |z|<|y|}\frac{d\sigma(z)}{|z|^{(n-2\al)q}}\right)^{\frac{q}{1-q}}d\sigma(y)\\
&\le c_1\,\frac{1}{|x|^{(n-2\al)q}}\left(\int_{|z|<|x|}\frac{d\sigma(z)}{|z|^{(n-2\al)q}}\right)^{\frac{q}{1-q}}\int_{|y|\ge|x|} \frac{d\sigma(y)}{|y|^{n-2\al}}\\
&+c_1\,\int_{|y|\ge|x|} \frac{1}{|y|^{n-2\al}|y|^{(n-2\al)q}}\left(\int_{|x|\le |z|<|y|}\frac{|z|^{(n-2\al)(1-q)}d\sigma(z)}{|z|^{n-2\al}}\right)^{\frac{q}{1-q}}
d\sigma(y)\\
&\le c_1\,\frac{1}{|x|^{(n-2\al)q}}\left(\int_{|z|<|x|}\frac{d\sigma(z)}{|z|^{(n-2\al)q}}\right)^{\frac{q}{1-q}}\int_{|y|\ge|x|} \frac{d\sigma(y)}{|y|^{n-2\al}}\\
&+c_1\,\left(\int_{|x|\le |z|}\frac{d\sigma(z)}{|z|^{n-2\al}}\right)^{\frac{q}{1-q}}\int_{|y|\ge|x|} \frac{d\sigma(y)}{|y|^{(n-2\al)}}.
\end{aligned}
\end{equation*} 
Using Young's inequality again, we arrive at 
\begin{equation*}
\begin{aligned} 
III & \le  c_1c_2\,\frac{1}{|x|^{n-2\al}}\left(\int_{|z|<|x|}\frac{d\sigma(z)}{|z|^{(n-2\al)q}}\right)^{\frac{1}{1-q}}\\
&+(c_1c_2+c_1) \left(\int_{|y|\ge |x|}\frac{d\sigma(y)}{|y|^{n-2\al}}\right)^{\frac{1}{1-q}}.
\end{aligned}
\end{equation*} 
Clearly, $$IV \le \left(\int_{|y|\ge |x|}\frac{d\sigma(y)}{|y|^{n-2\al}}\right)^{\frac{1}{1-q}}.$$
Clearly, $v(0)\ge \rat(v^qd\sigma)(0)$, if  $c$ is chosen large enough. Therefore, we obtain  $\rat(v^qd\sigma) \le v$. Using iterations  and the Monotone Convergence Theorem as above, we deduce the existence of a radial solution $u$ to the equation
$u=\rat(u^qd\sigma)$. Moreover, 
$$u(x) \le  c\,\left(\frac{1}{|x|^{n-2\al}}\left(\int_{|y|<|x|}\frac{d\sigma(y)}{|y|^{(n-2\al)q}}\right)^{\frac{1}{1-q}}
+\left(\int_{|y|\ge |x|}\frac{d\sigma(y)}{|y|^{n-2\al}}\right)^{\frac{1}{1-q}}\right),
$$
which completes the proof of  Theorem \ref{radpro}.
\end{proof}

We now characterize condition \eqref{newcond0} when $\sigma$ is radial and $\rat\sigma \not \equiv \infty$. Then for every $a>0$,
 $$\rat(x) \le c(a,\sigma), \quad  \text{ for } \,|x| \ge a.$$ 
Indeed, 
\begin{equation*}
\begin{aligned} 
\rat\sigma(x) & \ap  \frac{\sigma(B(0,|x|)}{|x|^{n-2\al}} + \int_{|y|\ge|x|} \frac{d\sigma(y)}{|y|^{n-2\al} }\\
& =\frac{\int_{|y|<a} d\sigma(y)}{|x|^{n-2\al}} +\frac{\int_{a\le|y|<|x|} d\sigma(y)}{|x|^{n-2\al}}+ \int_{|y| \ge |x|} \frac{d\sigma(y)}{|y|^{n-2\al} } \\
& \le \frac{\int_{|y|<a} d\sigma(y)}{|a|^{n-2\al}} +\int_{a\le|y|<|x|}\frac{ d\sigma(y)}{|y|^{n-2\al}}+ \int_{|y| \ge a} \frac{d\sigma(y)}{|y|^{n-2\al} } \\ 
&\le \frac{\int_{|y|<a} d\sigma(y)}{|a|^{n-2\al}} + 2\int_{|y| \ge a} \frac{d\sigma(y)}{|y|^{n-2\al} } .
\end{aligned}
\end{equation*} 

We observe that  $\lim_{|x| \to \infty} \rat\sigma(x) = 0$. If $\limsup_{|x| \to 0} \rat\sigma(x) < +\infty$, then by the above observation, we have $\rat\sigma \in L^{\infty}(\R^n)$. This implies that  condition \eqref{newcond0} holds. Therefore, we need to focus on the case where $\limsup_{|x| \to 0} \rat\sigma(x) = + \infty$. 
Let us set  
$$ \mathbf{K}\sigma(x) = \frac{1}{|x|^{n-2\al}}\left(\int_{|y|<|x|}\frac{d\sigma(y)}{|y|^{(n-2\al)q}}\right)^{\frac{1}{1-q}}, \quad x \neq 0.$$
Suppose that \eqref{newcond0} holds. Then by Theorem \ref{conthomothm}, there exists a solution $u$ to \eqref{subeq} such that $u \le c ( \rat\sigma + (\rat\sigma)^{\frac{1}{1-q}})$. On the other hand,   $ u \ge c \ \mathbf{K}\sigma$ by Theorem \ref{radpro}.
Therefore, if $\rat\sigma \not \equiv \infty$, then condition \eqref{newcond0} implies 
\begin{equation} \label{Kcond}
\mathbf{K}\sigma \le c\, \left( \rat\sigma + (\rat\sigma)^{\frac{1}{1-q}}\right) < \infty \, \text{ a.e}.
\end{equation}

Conversely, suppose that \eqref{Kcond} holds. Then by Theorem \ref{radpro} there exists a solution $u$ to \eqref{subeq} such that $u \le c \left(\mathbf{K}\sigma + (\rat\sigma)^{\frac{1}{1-q}}\right)$. Hence, 
$u \le c\, \left( \rat\sigma + (\rat\sigma)^{\frac{1}{1-q}}\right)$, and using \eqref{lowest} yields \eqref{newcond0}. Therefore, if $\rat\sigma \not \equiv \infty$, then \eqref{newcond0} holds if and only if \eqref{Kcond} holds. 

We next prove the following proposition.
\begin{prop} \label{radlemma} Let $0<q<1$, $n \ge 1$, and $0<2 \alpha<n$. 
Let $  \sigma \in M^+(\R^n)$ be radially symmetric. Suppose that $\limsup_{|x| \to 0} \rat\sigma(x) = +\infty$. Then there exists a constant $c>0$ such that \eqref{Kcond} holds if and only if \, $\int_{|y| \ge 1}\frac{d\sigma(y)}{|y|^{n-2\al} } < \infty$ \, and 
\begin{equation} \label{radcond}
\limsup_{|x| \to 0} \frac{\frac{1}{|x|^{(n-2\al)(1-q)}}\int_{B(0,|x|)} \frac{d\sigma(y)}{|y|^{q(n-2\al)}} }{ \int_{|y| \ge |x|} \frac{d\sigma(y)}{|y|^{n-2\al} }} < \infty. 
\end{equation}
\end{prop}
\begin{proof}
Suppose that $\int_{|y| \ge 1}\frac{d\sigma(y)}{|y|^{n-2\al} } < \infty$  and that \eqref{radcond} holds. Then there exists $ \delta, \, 0<\delta <1$, such that 
$$\mathbf{K}\sigma(x) \le c \, (\rat\sigma(x))^{\frac{1}{1-q}}, \quad \text{for all}  \,\,  0 < |x| < \delta.$$
For $ \delta \le |x| \le 1$, we have 
$$\frac{1}{|x|^{n-2\al}}\left(\int_{B(0,|x|)} \frac{d\sigma(y)}{|y|^{(n-2\al)q}}\right)^{\frac{1}{1-q}} \le \frac{1}{\delta^{n-2\al}}\left(\int_{B(0,1)} \frac{d\sigma(y)}{|y|^{(n-2\al)q}}\right)^{\frac{1}{1-q}}.$$
On the other hand,
\begin{equation*}
\begin{aligned} 
 \rat \sigma(x) & \ap \frac{1}{|x|^{n-2\al}}\int_{B(0,|x|)} d\sigma(y) + \int_{|y| \ge |x|}\frac{d\sigma(y)}{|y|^{n-2\al}} \\ & \ge \int_{B(0,\delta)} d\sigma(y) + \int_{|y| \ge 1}\frac{d\sigma(y)}{|y|^{n-2\al}}.
 \end{aligned}
\end{equation*} 
Therefore, there exists a constant $c=c(\delta,\sigma)$ such that 
$$\mathbf{K}\sigma(x) \le c \, \rat\sigma(x), \quad \text{when }  \delta \le |x| \le 1.$$
For $|x| \ge 1$, we have 
$$\int_{B(0,|x|)} \frac{d\sigma(y)}{|y|^{(n-2\al)q}} \le  \int_{0 \le |y| \le 1} \frac{d\sigma(y)}{|y|^{(n-2\al)q}} +  \int_{1\le |y| < |x|} \frac{d\sigma(y)}{|y|^{(n-2\al)q}} .$$
By H\"{o}lder's inequality,
$$\int_{1\le |y| < |x|} \frac{d\sigma(y)}{|y|^{(n-2\al)q}} \le \left(\int_{1\le |y| < |x|} d\sigma(y)\right)^{1-q}\left(\int_{1\le |y| < |x|} \frac{d\sigma(y)}{|y|^{n-2\al}}\right)^q.$$
Hence,
\begin{equation*}
\begin{aligned} 
 \left(\int_{1\le |y| < |x|} \frac{d\sigma(y)}{|y|^{(n-2\al)q}}\right)^{\frac{1}{1-q}} & \le \left(\int_{1\le |y| < |x|} d\sigma(y)\right)\left(\int_{1\le |y| < |x|} \frac{d\sigma(y)}{|y|^{n-2\al}}\right)^{\frac{q}{1-q}} \\
&\le \left(\int_{ |y| < |x|} d\sigma(y)\right)\left(\int_{ |y| \ge 1 } \frac{d\sigma(y)}{|y|^{n-2\al}}\right)^{\frac{q}{1-q}}.
\end{aligned}
\end{equation*}

Therefore, there exists a constant $c=c(\sigma, q)$ such that 
$$\mathbf{K}\sigma(x) \le c \, \rat\sigma(x) , \quad \text{ for } |x| \ge 1.$$
From this it follows 
 $$\mathbf{K}\sigma(x) \le c \, (\rat\sigma(x) + (\rat\sigma(x))^{\frac{1}{1-q}}) < \infty , \quad  x \neq 0.$$

Conversely, suppose that \eqref{Kcond} holds and $\limsup_{|x| \to 0}\rat\sigma(x) =+ \infty$. Then for $|x|$ small enough we have 
$$\rat\sigma(x) \le  (\rat\sigma(x))^{\frac{1}{1-q}}.$$
Consequently, $$\mathbf{K}\sigma(x) \le c \, (\rat\sigma(x))^{\frac{1}{1-q}} ,$$
when $|x|$ is close to $0$, which yields 
\begin{equation}\label{Kineq}
\begin{aligned} 
& \frac{1}{|x|^{(n-2\al)(1-q)}}\int_{|y|<|x|}\frac{d\sigma(y)}{|y|^{(n-2\al)q}} \\ & \le c\left(\frac{1}{|x|^{n-2\al}}\int_{B(0,|x|)}d\sigma(y)  
+ \int_{|y|\ge |x|} \frac{d\sigma(y)}{|y|^{n-2\al} } \right).
\end{aligned}
\end{equation}

We estimate
\begin{equation*}
\begin{aligned} 
\frac{1}{|x|^{n-2\al}}\int_{B(0,|x|)}d\sigma(y)& =
\frac{1}{|x|^{(n-2\al)(1-q)}|x|^{(n-2\al)q}}\int_{0\le|y| \le \delta}\frac{|y|^{(n-2\al)q}\,d\sigma(y)}{|y|^{(n-2\al)q}}\\
 & + \frac{1}{|x|^{n-2\al}}\int_{ \delta < |y| < |x|}\frac{|y|^{n-2\al}\,d\sigma(y)}{|y|^{n-2\al}} \\ &\le \frac{\delta^{(n-2\al)q}}{|x|^{(n-2\al)(1-q)}|x|^{(n-2\al)q}}\int_{0\le|y| \le \delta}\frac{d\sigma(y)}{|y|^{(n-2\al)q}} \\
&+ \int_{ \delta < |y| < |x|}\frac{d\sigma(y)}{|y|^{n-2\al}}.
\end{aligned}
\end{equation*}

Letting $\delta=c_1\,|x|$ where $c_1$ is small enough, we obtain
\begin{equation*}
\begin{aligned} 
\frac{1}{|x|^{n-2\al}}\int_{B(0,|x|)}d\sigma(y)& \le c_1^{(n-2\al)q}\,\frac{1}{|x|^{(n-2\al)(1-q)}}\int_{0\le|y| < |x|}\frac{d\sigma(y)}{|y|^{(n-2\al)q}} \\
& + \int_{ |y| > c_1\,|x|}\frac{d\sigma(y)}{|y|^{n-2\al}}.
\end{aligned}
\end{equation*}
Hence, by \eqref{Kineq}, we get
\begin{equation*}
\begin{aligned} 
  &\frac{1}{|x|^{(n-2\al)(1-q)}}\int_{B(0,|x|)} \frac{d\sigma(y)}{|y|^{q(n-2\al)}} \\&  \le  c\,c_1^{(n-2\al)q}\frac{1}{|x|^{(n-2\al)(1-q)}}\int_{B(0,|x|)} \frac{d\sigma(y)}{|y|^{q(n-2\al)}}\\
&+c \, \int_{|y| \ge c_1\,|x|} \frac{d\sigma(y)}{|y|^{n-2\al} }+ c \, \int_{|y| \ge |x|} \frac{d\sigma(y)}{|y|^{n-2\al} }. 
\end{aligned}
\end{equation*}

If $ c\,c_1^{(n-2\al)q} \le \frac{1}{2}$ and $c_1<1$, then
$$  \frac{1}{|x|^{(n-2\al)(1-q)}}\int_{B(0,|x|)} \frac{d\sigma(y)}{|y|^{q(n-2\al)}}  \le  4c \,\int_{|y| \ge c_1\,|x|} \frac{d\sigma(y)}{|y|^{n-2\al} }. $$
This implies that
$$  \frac{1}{|x|^{(n-2\al)(1-q)}}\int_{B(0,c_1|x|)} \frac{d\sigma(y)}{|y|^{q(n-2\al)}}  \le  4c \,\int_{|y| \ge c_1\,|x|} \frac{d\sigma(y)}{|y|^{n-2\al} }. $$
Letting $\tilde{x}=c_1\,x$, we obtain
$$\frac{1}{|\tilde{x}|^{(n-2\al)(1-q)}}\int_{B(0,\,|\tilde{x}|)} \frac{d\sigma(y)}{|y|^{q(n-2\al)}}  \le  \frac{4c}{c_1^{(n-2\al)(1-q)}}\,\int_{|y| \ge |\tilde{x}|} \frac{d\sigma(y)}{|y|^{n-2\al} }. $$
Therefore, \eqref{radcond} holds, which completes the proof of  Proposition \ref{radlemma}.
\end{proof}

In conclusion, we give a counter example showing that there exist  
radially symmetric $\sigma$ for which \eqref{radialcond} holds, and  so \eqref{subeq} has a nontrivial solution, but  condition  \eqref{newcond0},  and  consequently Brezis--Kamin type estimates   \eqref{twosdhomo} fail. Let 
\begin{equation}\label{countereg}
\sigma(y)= \left\{ \begin{array}{ll}
\frac{1}{|y|^s \log^{\beta}\frac{1}{|y|}} \, , \quad \text{ if } |y| < 1/2,\\
0 , \quad \text{ if } |y| \ge 1/2,
\end{array}\right.
\end{equation} 
where $s=(1-q)n+2 \alpha q$ and $\beta > 1$. Then $2\al<s<n$, 
since  $0<q<1$ and  $0<2\al<n$.

Clearly,   $\int_{|y|\ge 1}\frac{d\sigma(y)}{|y|^{n-2\al}}=0,$
and  
 $$\int_{|y|<1} \frac{d\sigma(y)}{|y|^{(n-2\al)q}}=\int_{|y|< \frac{1}{2}} \frac{dy}{|y|^{n} \log^{\beta}\frac{1}{|y|}} < \infty.$$ 
Hence, \eqref{radialcond} holds, and so \eqref{subeq} has a solution 
by Theorem~\ref{radpro}.

On the other hand, for $|x|$ small, 
$$\int_{B(0,|x|)} \frac{d\sigma(y)}{|y|^{(n-2\al)q}}=\int_{|y|< |x|} \frac{dy}{|y|^{n} \log^{\beta} \frac{1}{|y|}} \approx \log^{1-\beta}\frac{1}{|x|},$$ 
and
$$ \int_{|y|\ge|x|} \frac{d\sigma(y)}{|y|^{n-2\al}} \approx \frac{1}{|x|^{(n-2\al)(1-q)} \log^{\beta}\frac{1}{|x|}}.$$
Therefore,
$$\lim_{|x| \to 0} \frac{\frac{1}{|x|^{(n-2\al)(1-q)}}\int_{B(0,|x|)} \frac{d\sigma(y)}{|y|^{q(n-2\al)}}}{\int_{|y|\ge|x|} \frac{d\sigma(y)}{|y|^{n-2\al}}}
= \lim_{|x| \to 0}\log\frac{1}{|x|} = +\infty.$$
 
 It follows from Proposition~\ref{radlemma} that \eqref{Kcond} fails.  Then, 
 as shown above, \eqref{newcond0} fails as well. Hence, by Theorem~\ref{mainthm} with $p=2$, the upper estimate in \eqref{twosdhomo} is no longer true 
for any nontrivial solution $u$ 
of \eqref{subeq} in this case. We recall that the lower 
estimate in \eqref{twosdhomo} is always true for any nontrivial supersolution  $u$ (see \cite{CV2}).

\end{document}